\documentclass[a4paper,11pt,reqno]{amsart}
\usepackage{}
\usepackage{amssymb, amsmath, amsfonts, amscd}
\usepackage{mathrsfs, latexsym, array, longtable}
\usepackage[all,ps,cmtip]{xy}
\usepackage{xcolor, bm, enumitem}
\usepackage{url}
\usepackage{tikz}
\usepackage{verbatim}
\usepackage{graphicx}
\usepackage{cite}
\usepackage[all]{xy}
\usepackage{tikz-cd}
\DeclareMathOperator{\can}{can}

\title{On explicit birational geometry for minimal $n$-folds of canonical dimension $n-1$}

\author{Meng Chen}
\date{\today}
\address{\rm School of Mathematical Sciences, Fudan University, Shanghai 200433, China}
\email{mchen@fudan.edu.cn}

\author{Louis Esser}
\address{UCLA Mathematics Department,
Box 951555, Los Angeles, CA 90095-1555} \email{esserl@math.ucla.edu}

\author{Chengxi Wang}
\address{UCLA Mathematics Department,
Box 951555, Los Angeles, CA 90095-1555} \email{chwang@math.ucla.edu}

\thanks{M. Chen was supported by NSFC for Innovative Research Groups (\#12121001),
National Key Research and Development Program of China (\#2020YFA0713200) and NSFC (\#12071078,\#11731004). L. Esser was partially supported by NSF grant DMS-2054553.}

\newcommand{\bQ}{{\mathbb Q}}
\newcommand{\bP}{{\mathbb P}}
\newcommand{\roundup}[1]{\lceil{#1}\rceil}
\newcommand{\rounddown}[1]{\lfloor{#1}\rfloor}

\newcommand\Vol{\text{\rm Vol}}
\newcommand\lrw{\longrightarrow}

\newcommand\bZ{{\mathbb{Z}}}

\newcommand{\lsgeq}{\succcurlyeq}
\newcommand{\lsleq}{\preccurlyeq}

\newtheorem{thm}{Theorem}[section]

\newtheorem{cor}[thm]{Corollary}
\newtheorem{prop}[thm]{Proposition}
\newtheorem{claim}[thm]{Claim}

\theoremstyle{definition}
\newtheorem{defn}[thm]{Definition}

\newtheorem{question}[thm]{Question}
\newtheorem{exmp}[thm]{Example}

\newtheorem{rem}[thm]{Remark}

\begin{document}
\begin{abstract}  Let $n\geq 2$ be any integer.  We study the optimal lower bound $v_{n, n-i}$ of the canonical volume and the optimal upper bound $r_{n,n-i}$ of the canonical stability index for minimal projective $n$-folds of general type, which are canonically fibered by $i$-folds ($i=0,1$).  The results for $i = 0$, $v_{n,n}=2$ and $r_{n, n}=n+2$, are known to experts. In this article, we show that $v_{n,n-1}=\frac{6}{2n+(n \bmod 3)}$ and $r_{n,n-1}=\frac{1}{3}(5n+ 3 + (n \bmod 3))$. The machinery is applicable to all canonical dimensions $n-i$. 
\end{abstract}
\maketitle

\pagestyle{myheadings}
\markboth{\hfill M. Chen, L. Esser, and C. Wang\hfill}{\hfill On explicit birational geometry for minimal $n$-folds of canonical dimension $n-1$\hfill}
\numberwithin{equation}{section}

\section{Introduction}

In birational geometry, minimal varieties of general type form one of the basic building blocks of the minimal model program.  A crucial step toward classifying these varieties is to study the distribution of their birational invariants.  Since every smooth variety of general type has a minimal model, this will give the distribution of invariants in the smooth case as well. Throughout, we will work over any algebraically closed field of characteristic zero. 

We begin by recalling a few examples of birational invariants.  Let $n \geq 2$ be any integer and $X$ be a minimal projective $n$-fold of general type, so that $X$ has at worst $\bQ$-factorial terminal singularities and the canonical divisor $K_X$ is nef.  The $m$th {\it plurigenus} $P_m(X)$ of $X$ is the dimension of the space of sections $H^0(X,mK_X)$ for $m \geq 1$.  The {\it volume} of $X$ is a measure of the asymptotic growth of the plurigenera, defined as 

$$\Vol(X):= \limsup_{m \rightarrow \infty} \frac{n! P_m(X)}{m^n},$$
which is a positive rational number.  Since $X$ is minimal, so that $K_X$ is nef, the Riemann-Roch formula shows that $\Vol(X) = K_X^n$, the top intersection number of $K_X$.  We define the {\it canonical stability index} of $X$ by

$$r_s(X):=\text{min}\{p\in \bZ_{>0}| \varphi_{m,X}\ \text{is birational for all}\ m\geq p\},$$
where $\varphi_{m,X}$ denotes the $m$-canonical map of $X$.  For any dimension $n$, the $n$th {\it canonical stability index} is defined as:

$$r_n:={\text{max}}\{r_s(X)| X\ \text{is any minimal projective } n\text{-fold of general type}\}.$$

Similarly, we define

$$v_n:={\text{inf}}\{\Vol(X)| X\ \text{is any minimal projective } n\text{-fold of general type}\}.$$

A remarkable theorem of Hacon-M\textsuperscript{c}Kernan \cite{H-M06}, Takayama \cite{Ta} and Tsuji \cite{Ts} states that the canonical stability index is uniformly bounded in each dimension, so that $r_n<+\infty$.  As a consequence, we have $v_n>0$ (more precisely, $v_n \geq \frac{1}{(r_n)^n}$).
 Once boundedness is established, the natural next step in classification is to find explicit values for these bounds (see, for instance, Hacon-M\textsuperscript{c}Kernan
\cite[Problem 1.5, Question 1.6]{H-M06}).

We list some known results concerning $r_n$ and $v_n$:

\begin{itemize}
\item[$\circ$] By Bombieri \cite{Bom}, $r_2=5$ and $v_2=1$.

\item[$\circ$] By Iano-Fletcher \cite{F}, $r_3\geq 27$ and $v_3 \leq \frac{1}{420}$; by Chen-Chen \cite{CC1,CC2,CC3} and Chen \cite{57}, $r_3\leq 57$ and $v_3\geq \frac{1}{1680}$.

\item[$\circ$] By Esser-Totaro-Wang \cite{ETW}, for $n \geq 3$, $r_n > 2^{2^{(n-2)/2}}$ and $v_n < \frac{1}{2^{2^{n/2}}}$.
\end{itemize}

For a given minimal variety $X$ of general type, the {\it canonical dimension} of $X$ is defined as $d_1 := \dim\overline{\varphi_{1}(X)}$. For $1\leq i<n$, we define $r_{n,n-i}$ and $v_{n,n-i}$ to be the maximal canonical stability index and the infimum of volumes, respectively, among all minimal $n$-folds of general type with canonical dimension $n-i$.  Similarly, $r_{n,0}$ and $v_{n,0}$ (resp. $r_{n,-\infty}$ and $v_{n,-\infty}$) are the corresponding values for minimal $n$-folds with geometric genus $p_g = 1$ (resp. $p_g = 0$).  With these definitions,

\begin{eqnarray*}
r_n&=&\text{max}\{r_{n,j}| -\infty\leq j\leq n\} \text{ and} \\
v_n&=&\text{min}\{v_{n,j}| -\infty\leq j\leq n\}.
\end{eqnarray*}

When $X$ has canonical dimension $n-i$ for small $i$, we can analyze it using its canonical fibration by $i$-folds.  In particular, the values $r_{n,n-1}$ have long been studied with applications to the classification of algebraic varieties, beginning with surfaces. In this paper, we study values of $r_{n,j}$ and $v_{n,j}$ for $j\geq n-1$ and prove the following theorem:


\begin{thm}\label{main} Let $n\geq 2$ be any integer.  The following statements hold:
\begin{itemize}
\item[(1)] $v_{n,n}=2$ and $r_{n,n}=n+2$;
\item[(2)]
$v_{n,n-1}=\frac{6}{2n+(n \bmod 3)}$ and
$r_{n,n-1}=\frac{1}{3}(5n+3+(n \bmod 3));$
\end{itemize}
where
$``n \bmod 3"$ is the minimal non-negative residue of $n$ modulo $3$.
\end{thm}

In each dimension, we will present optimal examples achieving the bounds in (1) and (2), which will be weighted projective hypersurfaces of general type.  For an introduction to weighted projective hypersurfaces, see \cite[Section 2]{ETW}.

\begin{rem}  Some special values of $v_{n,j}$ and $r_{n,j}$ are already known:
\begin{itemize}
\item[$\circ$] The surface case is due to Bombieri \cite{Bom};
\item[$\circ$] The threefold case is proved by one of the authors (Chen) \cite{MA, IJM} ;
\item[$\circ$] The value of $v_{n,n}$ is originally due to Kobayashi \cite{Kobayashi}; 
\item[$\circ$]  The values of $v_{4,3}$, $v_{4,2}$, $v_{5,4}$ and $v_{5,3}$ are already obtained in Chen-Jiang-Li \cite{CJL}.
\end{itemize}
\end{rem}

\bigskip
Throughout, we make use of the following notation:
\begin{itemize}
\item[\textordmasculine] For two $\bQ$-divisors $D_1$ and $D_2$, $D_1 \sim_{\bQ} D_2$ (resp., $D_1 \equiv D_2$) means that $D_1$ is $\bQ$-linearly equivalent (resp., numerically equivalent) to $D_2$.  Similarly, $D_1\geq_{\text{num}}D_2$ (resp.,  $D_1\geq_{\bQ}D_2$) means that $D_1-D_2$ is numerically (resp., $\bQ$-linearly) equivalent to an effective $\bQ$-divisor.
\item[\textordmasculine] For two linear systems $|A|$ and $|B|$ on a variety, we write $|A|\lsgeq |B|$ (or symmetrically, $|B|\lsleq |A|$) if there exists an effective divisor $F$ such that
     $$|A|\supseteq |B|+F.$$
\end{itemize}

\section{A key technical theorem}

\begin{thm}\label{kt} Let $X$ be a minimal projective $n$-fold of general type.  Let
$X'$ be a nonsingular projective variety and  $\pi: X'\lrw X$ be a birational morphism. Assume that there exists a chain of smooth subvarieties of general type
$$Z_1\subset Z_2\subset \cdots \subset Z_{n-1}\subset Z_n=X'$$
with $\dim Z_j=j$ for $j=1,2,\ldots, n-1$.  Suppose that the following conditions hold:
\begin{itemize}
\item[(i)] $\pi^*(K_X)|_{Z_i}$ is big for each $i=2,\ldots, n-1$;
\item[(ii)] 
$\pi^*(K_X)|_{Z_i}\equiv \beta_i Z_{i-1}+\triangle_{i-1}$
where $\beta_i$ is a positive rational number, $\triangle_{i-1}$ is an effective $\bQ$-divisor on $Z_{i}$ and 
$$Z_{i-1}\not\in \text{Supp}(\triangle_{i-1}+\sum_{j\geq i}{\triangle_j}|_{Z_i})$$  for each $i=2,\ldots, n$;
\item[(iii)]  the number  $\xi:=\big(\pi^*(K_X)\cdot Z_1\big)>0$. 
\end{itemize}
Then each of the following statements holds:
\begin{itemize}
\item[(1)]  For any integer $m$ with $\alpha_m:=(m-1-\sum_{i=2}^n\frac{1}{\beta_i})\xi>1$, the inequality
\begin{equation} m\xi\geq (2g(Z_1)-2)+\roundup{\alpha_m}\label{Ine1}
\end{equation}holds. In particular, under this situation, one has
\begin{equation}\xi\geq \frac{2g(Z_1)-2}{1+\sum_{i=2}^n\frac{1}{\beta_i}}.\label{Ine2}\end{equation}
\item[(2)] The canonical volume of $X$ has the lower bound:
\begin{equation}K_X^n\geq \beta_2\beta_3\cdots\beta_{n} \xi.\label{Ine3}\end{equation}
\end{itemize}
\end{thm}
\begin{proof}  First of all, if we replace $X'$ by any nonsingular birational model (which dominates $X'$) for which no $Z_i$ is contained in the image of the exceptional locus, then we may replace each $Z_i$ by its strict transform.  After performing this replacement, the assumption still holds for the same values of $\beta_i$ and the number $\xi$ remains unchanged by the projection formula.  Hence, fixing an effective divisor $K_1\sim K_X$ at the beginning, we may and do assume that, on $X'$, the union of the divisor $\pi^*(K_1)|_{Z_i}$, $\triangle_{i-1}$  (for all $i=2,\dots, n$) and exceptional divisors has simple normal crossing supports.

Suppose $m>1+\sum_{i=2}^n\frac{1}{\beta_i}$. Since the $\bQ$-divisor
$$(m-1)\pi^*(K_X)|_{Z_n}-\frac{1}{\beta_n}\triangle_{n-1}-Z_{n-1}\equiv \left(m-1-\frac{1}{\beta_n}\right)\pi^*(K_X)|_{Z_n}$$
is nef and big, and the fractional part has snc supports, the Kawamata-Viehweg vanishing theorem \cite{KV,VV} implies:
{\footnotesize \begin{eqnarray}
|mK_{X'}||_{Z_{n-1}}&\lsgeq &|K_{X'}+\roundup{(m-1)\pi^*(K_X)-\frac{1}{\beta_n}\triangle_{n-1}}||_{Z_{n-1}}\notag\\
&\lsgeq &|K_{Z_{n-1}}+\roundup{\big((m-1)\pi^*(K_X)-Z_{n-1}-\frac{1}{\beta_n}\triangle_{n-1}\big)|_{Z_{n-1}}}|. \label{e1}
\end{eqnarray}}
By induction, one clearly gets, for $i=2,\ldots, n-1$, that
{\footnotesize \begin{eqnarray}
|mK_{X'}||_{Z_{i-1}}&\lsgeq &|K_{Z_{i}}+\roundup{(m-1)\pi^*(K_X)|_{Z_i}-
\sum_{l=i+1}^n\big(Z_{l-1}+\frac{1}{\beta_l}\triangle_{l-1}\big)|_{Z_{i}}-\frac{1}{\beta_i}\triangle_{i-1}}|_{Z_{i-1}}\notag\\
&\lsgeq &|K_{Z_{i-1}}+\roundup{(m-1)\pi^*(K_X)|_{Z_{i-1}}-\sum_{l=i}^n\big(Z_{l-1}+\frac{1}{\beta_l}\triangle_{l-1}\big)|_{Z_{i-1}}}|. \label{e2}
\end{eqnarray}}
By \eqref{e1} and \eqref{e2} while repeatedly using \cite[Lemma 2.7]{MPCPS}, we have
\begin{equation}
\label{degree}
{M_m}|_{Z_1}\geq  \text{Mov}|K_{Z_1}+\roundup{Q_m}|,
\end{equation}
where $Q_m:=(m-1)\pi^*(K_X)|_{Z_1}-\sum_{i=2}^n\big(Z_{i-1}+\frac{1}{\beta_i}\triangle_{i-1}\big)|_{Z_1}$ and $M_m:=\text{Mov}|mK_{X'}|$.
Note that, by Condition (ii), the divisor $Z_{i-1}|_{Z_{i-1}}$ never appears in fractional part of all above $\bQ$-divisors. 

Whenever $\alpha_m = \deg(Q_m) >1$,  $Q_m\equiv (m-1-\sum_{i=2}^n\frac{1}{\beta_i})\pi^*(K_X)|_{Z_1}$ is nef and big, and $\text{Mov}|K_{Z_1}+\roundup{Q_m}|=K_{Z_1}+\roundup{Q_m}$, because divisors of degree at least $2g(Z_1)$ are base point free. Since $m \pi^*(K_X) \geq M_m$ and $\deg(\roundup{Q_m})\geq \deg(Q_m)$, taking degrees on both sides of \eqref{degree} proves \eqref{Ine1} in (1). 

Taking a sufficiently large integer $m'$ so that $\alpha_{m'}>1$, \eqref{Ine2} is a consequence of \eqref{Ine1}. Statement (2) then directly follows from the assumptions and the fact that $\pi^*(K_X)$ is nef.
\end{proof}

\begin{defn}\cite[Definition 2.3]{MZ2007}
A {\it generic irreducible element} $S$ of a movable linear system $|N|$ on a variety $Z$ is a generic irreducible component in a general member of $|N|$. Thus, it is a general member of $|N|$ whenever $\dim\varphi_{|N|}(V)\geq 2$. If $|N|$ is composed with a pencil, i.e., $\dim\varphi_{|N|}(V)= 1$, one has $N \equiv tS$ for some integer $t \geq 1$.
\end{defn}

\begin{defn}\cite[Definition 2.6]{CC3}
Let $|N|$ be a movable linear system on a variety $Z$. Pick two different generic irreducible elements $S'$, $S''$ in $|N|$. We say that a linear system $|M|$ (resp. a rational map $\varphi$ corresponding to a linear system) {\it distinguishes} $S'$ and $S''$ if $\varphi_{|M|}(S')\neq \varphi_{|M|}(S'')$ (resp. if $\varphi(S')\neq \varphi(S'')$). 
\end{defn}

In proving birationality of the rational map $\varphi_{\Lambda}$, where $\Lambda\subset |L|$ and $L$ a divisor on a projective variety $Z$, we tacitly use the following rule: 

\begin{quote}
Let $|M|$ be a base point free linear system on $Z$ and denote by $S$ a generic irreducible element of $|M|$. If $\varphi_{\Lambda}$ distinguishes different generic irreducible elements of $|M|$ and $\varphi_{\Lambda}|_S$ is birational, then $\varphi_{\Lambda}$ is birational. Conversely, if $\varphi_{\Lambda}$ is birational, it is clear that, for any base point free linear system $|M|$, $\varphi_{\Lambda}$ distinguishes different generic irreducible elements of $|M|$ and that $\varphi_{\Lambda}|_S$ is birational. 
\end{quote}

We'll also make use of the following \cite[Section 2.7]{CC2}:

\begin{prop}[Birationality principle]
Let $D$ and $M$ be two divisors on a smooth projective variety $Z$. Assume that $|M|$ is base point free. Take the Stein factorization of $\varphi_{|M|}: Z \xrightarrow{f} W \rightarrow \bP^{h^0(Z,M)-1}$, where $f$ is a fibration onto a normal variety $W$. The rational map $\varphi_{|D|}$ is birational onto its image if one of the following conditions is satisfied:
\begin{itemize}
\item $\dim$ $\varphi_{|M|}(Z)\geq 2$, $|D-M|\neq \emptyset$ and $\varphi_{|D|}|_S$ is birational for a general member $S$ of $|M|$
\item  $\dim$ $\varphi_{|M|}(Z)=1$, $\varphi_{|D|}$ distinguishes general fibers of $f$ and ${\varphi_{|D|}}|_S$ is birational for a general fiber $S$ of $f$.
\end{itemize} 
\end{prop}

In practice, in order to prove the birationality of $\varphi_{|D|}$, it is sufficient to find two sublinear systems $\Lambda_i\subset |D|$ ($i=1,2$) such that $\varphi_{\Lambda_1}$ distinguishes different generic irreducible elements of $|M|$ while ${\varphi_{\Lambda_2}}|_S$ is birational for a generic irreducible element $S$ of $|M|$.

\begin{cor}\label{BIRAT} Keep the same assumptions and notation as in Theorem \ref{kt}. Let $m$ be a positive integer. Assume that the following conditions are satisfied:

For each $i=2,\ldots,n$,  
\begin{itemize}
\item[(1)]  there exists a base point free linear system $|N_i|$ on $Z_i$ such that $Z_{i-1}$ is the generic irreducible element of  $|N_i|$;
\item[(2)]  the linear system $|mK_{X'}||_{Z_i}$ distinguishes different generic irreducible elements of $|N_i|$;
\item[(3)] $\alpha_m>2$.
\end{itemize}
Then $\varphi_{m,X}$ is birational onto its image.
\end{cor}
\begin{proof} We refer to the proof of Theorem \ref{kt}.  Applying the birationality principle repeatedly, by Assumptions (1) - (2) , Relation \eqref{e1} and Relation \eqref{e2}, $\varphi_{m,X}$ is birational if and only if so is $\varphi_{m,X'}|_{Z_{n-1}}$, by inductions, $\cdots$, if and only if so is $\varphi_{m,X'}|_{Z_2}$, if and only if so is $\varphi_{m,X'}|_{Z_1}$. 
Now we know that $\varphi_{|K_{Z_1}+\roundup{Q_m}|}$ is birational as $\alpha_m>2$. 
{}From the proof of Theorem \ref{kt}, we have seen that $|K_{Z_1}+\roundup{Q_m}|\lsleq |mK_{X'}||_{Z_1}$. Hence $\varphi_{m,X'}|_{Z_1}$ is birational. We are done. 
\end{proof}

\section{The canonical map of varieties of general type}

\subsection{Set up for $\varphi_{1,X}$}\label{set} \

Let $X$ be a minimal projective $n$-fold ($n\geq 2$) of general type with $p_g(X)\geq 2$ and canonical dimension $d_1$. Clearly we have $1\leq d_1\leq n$.
Let $\pi:X'\lrw X$ be a succession of blow-ups along nonsingular centers such that $|M_1|$, the linear system of $\text{Mov}|K_{X'}|$, is base point free and that the union of the fixed part of $|K_{X'}|$ and exceptional divisors of $\pi$ has simple normal crossing supports. Take $g_1:= \varphi_{1,X}\circ \pi$, which is a projective morphism.  Denote by $f:X'\lrw \Gamma$ be the induced fibration after taking the Stein factorization of $g_1$. So $\dim \Gamma=d_1$.

Set $|N_n|:= |M_1|$ and $Z_{n}:= X'$.  For any integer $k$ with $n-d_1+2\leq k\leq n$, inductively, take $Z_{k-1}$ to be the generic irreducible element of $|N_k|$ and set
$|N_{k-1}|:= |{M_1}|_{Z_{k-1}}|$. By Bertini's theorem, $|N_{k-1}|$ is base point free and is not composed with a pencil for each $k$. We have defined the following chain of nonsingular subvarieties:
$$Z_{n-d_1+1}\subset \cdots \subset Z_{n-1}\subset Z_n=X'.$$

Let us review some established inequalities in Chen-Jiang-Li \cite[Step 0, Proof of Theorem 5.1]{CJL}. Remodify $\pi$ (if necessary) such that, for $1\leq j\leq d_1-1$, $Z_{n-j}$ dominates a minimal model $Z_{n-j,0}$ and, in particular, there is a birational morphism $\pi_{Z_{n-j,0}}:Z_{n-j}\lrw Z_{n-j,0}$.  By induction, for $1\leq j\leq d_1-1$, one has
\begin{equation}
h^0(X', {M_1}|_{Z_{n-j}})\geq p_g(X')-j \label{h0} \end{equation}
Similar to Chen-Jiang-Li \cite[Inequality (5.4)]{CJL}, we have
\begin{equation}
\pi^*(K_X)|_{Z_{n-j}}\geq \frac{1}{j+1}\pi_{Z_{n-j}}^*(K_{Z_{n-j,0}})
\end{equation}
for $1\leq j\leq d_1-1$.  In particular, we have
\begin{equation}
\label{minmodelineq}
\pi^*(K_X)|_{Z_{n-d_1+1}}\geq \frac{1}{d_1}\pi_{Z_{n-d_1+1}}^*(K_{Z_{n-d_1+1,0}}).
\end{equation}

\subsection{Proof of Theorem \ref{main}: the case $d_1=n$}\

Theorem \ref{kt} and Corollary \ref{BIRAT} naturally apply to this very special case. We have already defined the chain of subvarieties:
$$Z_1\subset Z_2\subset \cdots\subset Z_{n-1}\subset Z_n=X'.$$
By construction, we have $\beta_i\geq 1$ for each $i$ with $2\leq i\leq n$.
Since $|{M_1}|_{Z_2}|$ is base point free and is not composed with a pencil, we have
$$\xi\geq (M_1\cdot Z_1)_{X'}\geq (Z_1|_{Z_2})^2\geq  2.$$
Hence, by \eqref{Ine3},  $v_{n,n}\geq 2$.

Take any integer $m>n+1$. Since $p_g(X)>0$ , the proof of Theorem \ref{kt} simply implies 
$|mK_{X'}||_{Z_i}\lsgeq |N_i|$  for $2\leq i\leq n$. Because each linear system $|N_i|$ is not composed with a pencil, the linear system $|mK_{X'}||_{Z_i}$ naturally distinguishes generic irreducible elements of $|N_i|$. Conditions Corollary \ref{BIRAT}(1) and (2) are therefore automatically satisfied.  Since $\alpha_m>\xi\geq 2$, we have proved that $r_s(X)\leq n+2$ by Corollary \ref{BIRAT}. Hence $r_{n,n}\leq n+2$.

\begin{exmp} For any integer $n\geq 2$, the general hypersurface
$$X = V_{2(n+2)}\subseteq \bP(1^{(n+1)}, n+2)$$
(of degree $2(n+2)$) is a smooth canonical $n$-fold with $\omega_X\cong \mathcal{O}_X(1)$, $\Vol(X)=2$ and $r_s(X)=n+2$.  Clearly, $X$ has canonical dimension $d_1=n$.
\end{exmp}

So we conclude that Theorem \ref{kt} (1) is true, i.e. $v_{n,n}=2$ and $r_{n,n}=n+2$.  Conversely, if $X$ has canonical dimension $n$ and $\Vol(X)=v_{n,n} = 2$, $X$ is a double cover of $\bP^n$ \cite[Proposition 2.5]{Kobayashi}.

\subsection{Further setting for the case $d_1<n$}\label{d1n}\

We take $Z_{n-d_1}$ to be the generic irreducible element of $|N_{n-d_1+1}|:=|{M_1}|_{Z_{n-d_1+1}}|$. Since $|N_{n-d_1+1}|$ is composed with a pencil and by \eqref{h0}, we have
$$N_{n-d_1+1}\equiv \beta_{n-d_1+1}Z_{n-d_1}$$
where $\beta_{n-d_1+1}\geq h^0({M_1}|_{Z_{n-d_1+1}})-1\geq p_g(X)-d_1$.

In addition, by our definitions, we have $\beta_i\geq 1$ for each $i$ with $n-d_1+2\leq i\leq n$.

\section{Proof of Theorem \ref{main}: the case $d_1=n-1$}

We have defined the chain $\{Z_i|i=1,\ldots,n\}$ according to Subsections \ref{set} and \ref{d1n}.  Note that $|N_2|$ is composed with a pencil of curves on the smooth surface $Z_2$. By our definition, $C:= Z_1$ is the generic irreducible element of $|N_2|$. Furthermore, we have $\beta_i\geq 1$ for $3\leq i\leq n$ and $\beta_2\geq p_g(X)-n+1\geq 1$. Since $Z_1$ moves in an algebraic family and $\pi^*(K_X)|_{Z_2}$ is nef and big, we have $\xi>0$ by the Hodge index theorem.

\subsection{Volume estimation}\label{sp}\

Noting that $g(Z_1)\geq 2$, by \eqref{Ine2}, we have $\xi\geq \frac{2}{n}$. We optimize the estimation by distinguishing three exclusive cases:
\smallskip

{\bf Case \ref{sp}.1}. $n=3k+2$, $k\geq 0$.

Take $m_1=\rounddown{\frac{9}{2}k+4}$.  Then, since $\alpha_{m_1}>1$,
 we have
$$\xi\geq \frac{8}{9k+8}=\frac{2\cdot4}{(2\cdot 4+1)k+2\cdot 4}$$
by \eqref{Ine1}. We will repeatedly use \eqref{Ine1} to optimize the estimation of $\xi$.
In fact,  suppose that we already know
$$\xi\geq\frac{2\cdot 4^l}{(2\cdot 4^l+1)k+2\cdot 4^l}$$
for an integer $l>0$.   Take
$$m_l=\rounddown{\frac{(2\cdot 4^l+1)k}{2\cdot 4^l}}+3k+4.$$
Then we see $\alpha_{m_l}>1$ and \eqref{Ine1} implies
$$\xi\geq \frac{2\cdot 4^{l+1}}{(2\cdot 4^{l+1}+1)k+2\cdot 4^{l+1}}.$$
Taking the limit, as $l\mapsto +\infty$, we have
\begin{equation}
\xi\geq \frac{1}{k+1}=\frac{3}{n+1}. \label{x1}
\end{equation}
\smallskip

{\bf Case \ref{sp}.2}. $n=3k+1$, $k\geq 1$.

Take $m_1=\rounddown{\frac{9k+1}{2}+2}$. Then, since $\alpha_{m_1}>1$,  we have
$$\xi\geq \frac{8}{9k+5}=\frac{2\cdot 4}{(2\cdot 4+1)k+4+1}$$
by \eqref{Ine1}.  Assume that we have shown
$$\xi\geq \frac{2\cdot 4^l}{(2\cdot 4^l+1)k+(4^l+4^{l-1}+\cdots+1)}$$
for some integer $l>0$.
Take $m_l=\rounddown{\frac{(2\cdot 4^l+1)k+(4^l+4^{l-1}+\cdots+1)}{2\cdot 4^l}}+3k+2$. Then, since $\alpha_{m_l}>1$,
 we get
$$\xi\geq \frac{2\cdot 4^{l+1}}{(2\cdot 4^{l+1}+1)k+(4^{l+1}+4^l+\cdots+1)}.$$
Taking the limit, as $l\mapsto +\infty$, we have $\xi\geq \frac{1}{k+2/3}$.
Finally, taking $m=4k+2$, then we get
\begin{equation}
\xi\geq \frac{2}{2k+1}=\frac{6}{2n+1}.\label{x2}
\end{equation}
\smallskip

{\bf Case \ref{sp}.3}. $n=3k$, $k\geq 1$.

Take $m_1=\rounddown{\frac{9k}{2}+1}$. Then, since $\alpha_{m_1}>1$, we get
$$\xi\geq \frac{8}{9k+2}>\frac{2\cdot 4}{(2\cdot 4+1)k+4}$$
by \eqref{Ine1}.
Suppose that, for some $l>0$, we have already shown
$$\xi\geq \frac{2\cdot 4^l}{(2\cdot 4^l+1)k+(4^l+\cdots+4)}.$$
Take
$$m_l=\rounddown{\frac{(2\cdot 4^l+1)k+(4^l+\cdots+4)}{2\cdot 4^l}}+3k+2.$$
Then, since $\alpha_{m_l}>1$, we have
$$\xi\geq \frac{2\cdot 4^{l+1}}{(2\cdot 4^{l+1}+1)k+(4^{l+1}+\cdots+4)}$$
by \eqref{Ine1}. Taking the limit, as $l\mapsto +\infty$, we have $\xi\geq \frac{1}{k+2/3}$.
Take $m=4k+1$. Then $\xi\geq \frac{4}{4k+1}=\frac{1}{k+1/4}$.

We continue the optimization for the lower bound of $\xi$. Assume that, for some integer $t\geq 4$, we already know $\xi\geq \frac{1}{k+1/t}$. Take $m_t=(t+1)k+1$. Then, since $\alpha_{m_t}>t-2\geq 2$, we have
$$\xi\geq \frac{t+1}{(t+1)k+1}=\frac{1}{k+1/(t+1)}$$
by \eqref{Ine1}.  Taking the limit, as $t\mapsto +\infty$, then we have
\begin{equation}\xi\geq \frac{1}{k}=\frac{3}{n}.\label{x3}
\end{equation}

By \eqref{x1}, \eqref{x2} and \eqref{x3}, we have proved the following:
\begin{cor}\label{41} For any integer $n\geq 2$, $v_{n,n-1}\geq \frac{6}{2n+(n \bmod 3)}$.
\end{cor}

\subsection{The canonical stability index}\label{sq}

As we have seen, for $3\leq i\leq n$, $|N_i|$ is not composed with a pencil. Applying the Kawamata-Viehweg vanishing theorem in the similar way to that of \eqref{e1} and \eqref{e2}, for any $m>n-2$ and any $j$ with $3\leq j\leq n-1$, we have
\begin{eqnarray}
|mK_{X'}||_{Z_j}&\lsgeq & |K_{X'}+\roundup{(m-n+2)\pi^*(K_X)}+(n-3)M_1||_{Z_{j}}\notag\\
&\lsgeq& |{M_1}|_{Z_{j}}|. \label{ee1}
\end{eqnarray}
Hence $|mK_{X'}||_{Z_j}$ distinguishes different generic irreducible elements of $|N_j|$ whenever $m>n-2$ and $3\leq j\leq n$.

We study the situation on $|N_2|$, which is composed with a pencil of curves. If $|N_2|$ is composed with a rational pencil of curves, as long as $m>n-1$, $|mK_{X'}||_{Z_2}$ distinguishes different generic irreducible elements of $|N_2|$ since
$|mK_{X'}||_{Z_2}\lsgeq |N_2|$.  If $|N_2|$ is composed with a non-rational pencil of curves, picking two different generic irreducible elements $C_1$, $C_2$ of $|N_2|$, then  $(M_1|_{Z_2}-C_1-C_2)$ is nef and
$$(M_1|_{Z_2}-C_i)|_{C_{i}}\sim 0$$
for $i=1,\ 2$.  Hence, by the vanishing theorem and for any $m>n$, one has the relation:
\begin{eqnarray*}
|mK_{X'}||_{Z_2}&\lsgeq&|K_{X'}+\roundup{(m-n)\pi^*(K_X)}+(n-1)M_1|\\
&\lsgeq&|K_{Z_2}+\roundup{(m-n)\pi^*(K_X)|_{Z_2}}+{M_1}|_{Z_2}|
\end{eqnarray*}
and the surjective map:
\begin{eqnarray*}
&&H^0(Z_2, K_{Z_2}+\roundup{(m-n)\pi^*(K_X)|_{Z_2}}+{M_1}|_{Z_2})\\
&\lrw& H^0(C_1, K_{C_1}+D_1)\oplus H^0(C_2, K_{C_2}+D_2)
\end{eqnarray*}
where
$$D_i:= (\roundup{(m-n)\pi^*(K_X)|_{Z_2}}+{M_1}|_{Z_2}-C_i)|_{C_i}=\roundup{(m-n)\pi^*(K_X)|_{Z_2}}|_{C_i}$$ for $i=1,\ 2$. Note that $H^0(C_i, K_{C_i}+D_i)\neq 0$ since $\deg(D_i)>0$. Therefore $|mK_{X'}|$ can distinguish different generic irreducible elements of $|N_2|$ in this case.
In summary, for any $m>n$, Conditions (1) and (2) of Corollary \ref{BIRAT} are satisfied.

We have proved that $\xi\geq \frac{6}{2n+(n \bmod 3)}$. For any $m\geq \frac{5n+3 + (n \bmod 3)}{3}$, we have $\alpha_m>2$. By Corollary \ref{BIRAT}, whenever
$m\geq \frac{5n+3 + (n \bmod 3)}{3}$, $\varphi_{m,X}$ is birational onto its image.  We have proved the following:

\begin{cor}\label{42} For any integer $n\geq 2$, $r_{n,n-1}\leq \frac{5n+3+ (n \bmod 3)}{3}$.
\end{cor}

\subsection{Examples}
\label{n-1_ex}

\begin{exmp} (cf. \cite[Example 6.6(1)]{CJL})\label{43} Let $n=3k+2$. The general hypersurface
$X=V_{10k+10}\subset \bP(1^{(3k+2)}, 2(k+1), 5(k+1))$ of degree $10(k+1)$ has at worst canonical singularities and is a minimal $n$-fold with $\omega_X = \mathcal{O}_X(1)$, canonical dimension $n-1$, and volume $\Vol(X)=\frac{1}{k+1}=\frac{3}{n+1}=\frac{6}{2n+(n \bmod 3)}$. Furthermore, it is clear that $r_s(X)=5k+5=\frac{5n+3 + (n \bmod 3)}{3}$.
\end{exmp}

\begin{exmp} (cf. \cite[Example 6.6(2)]{CJL})\label{44} Let $n=3k+1$. The general hypersurface
$X=V_{10k+6}\subset \bP(1^{(3k+1)}, 2k+1, 5k+3)$ of degree $10k+6$ has at worst canonical singularities and is a minimal $n$-fold with $\omega_X = \mathcal{O}_X(1)$, canonical dimension $n-1$, and volume
$\Vol(X)=\frac{6}{2n+1}=\frac{6}{2n+(n \bmod 3)}$. Furthermore, it is clear that $r_s(X)=5k+3=\frac{5n+3 + (n \bmod 3)}{3}$.
\end{exmp}

\begin{exmp}
\label{45}
Let $n = 3k$.  The general hypersurface
$X = V_{20k+2} \subset \bP(1,2^{(3k-1)},4k,10k+1)$ of degree $20k+2$ has at worst canonical singularities and is a minimal $n$-fold with $\omega_X = \mathcal{O}_X(2)$, canonical dimension $n-1$, and volume $\Vol(X) = \frac{3}{n} = \frac{6n}{2n + (n \bmod 3)}.$  Furthermore, it is clear that $r_s(X) = 5k+1 = \frac{5n+3 + (n \bmod 3)}{3}$.
\end{exmp}


Theorem \ref{main}(2) follows directly from Corollary \ref{41}, Corollary \ref{42}, Example \ref{43}, Example \ref{44} and Example \ref{45}.

\section{Geometry of Optimal Examples}\label{GeoP}

In this section, we summarize an alternate approach proving the volume bound in Theorem \ref{main}(2). This method also shows that any $X$ realizing the minimal value of $v_{n,n-1}$ must strongly resemble one of the weighted projective hypersurfaces in Section \ref{n-1_ex}. We omit some of the proof details and continue to use the notation of Section \ref{set}.

\begin{claim}
Let $X$ be a minimal variety of general type with dimension $n \geq 3$ and canonical dimension $n-1$ such that 
\begin{equation}
\label{Volmin}
\Vol(X) = \frac{6}{2n + (n \bmod 3)}.
\end{equation}
Then the canonical map of $X$ is a dominant rational map to $\bP^{n-1}$ with general fiber a genus $2$ curve.  Furthermore, the general surface $Z_2$ has minimal volume in the sense of Proposition \ref{extremesurf} below.  Its minimal model $Z_{2,0}$ contains a tree of rational curves contracted by the morphism from $X'$ to its canonical model, described in Cases \ref{GeoP}.1, \ref{GeoP}.2, and \ref{GeoP}.3.
\end{claim}

\begin{proof}

Let $X$ be a minimal variety realizing \eqref{Volmin} as above.  We immediately see that $g(Z_1) = 2$, or else $\xi \geq \frac{4}{n}$ by \eqref{Ine2}.  Further, $p_g(X) = n$, or else $\beta_2 \geq 2$, so that $K_X^n \geq 2 \cdot \frac{2}{n} = \frac{4}{n}$ by \eqref{Ine3}.  The image $\overline{\varphi_1(X)}$ of the canonical map is then a subvariety of $\bP^{n-1}$ of dimension $n-1$, so $\overline{\varphi_1(X)} = \bP^{n-1}$.  Therefore, the canonical map is a dominant rational map to $\bP^{n-1}$ whose general fiber is a smooth genus 2 curve. 

We can also say a great deal about $Z_2$ and the properties of the birational contraction $\pi:X' \rightarrow X_{\can}$, restricted to $Z_2$.  Here $X_{\can} = \mathrm{Proj} \bigoplus_{m=1}^{\infty} H^0(X,mK_X)$ is the canonical model of the variety $X$ (equivalently, of $X'$) and $K_{X_{\can}}^n = K_X^n$. 

The key idea is to find a commutative diagram
\begin{center}
\label{factor}
\begin{tikzcd}
  X' \arrow[rr, "\pi_{\can}"] & & X_{\can} \\
Z_2 \arrow[u, hook] \arrow[r, "g"] & Z_{2,0} \arrow[r, "f"]  & S \arrow[u, "\psi"],
\end{tikzcd}
\end{center}
where $g: Z_2 \rightarrow Z_{2,0}$ is the minimal model, $S$ is a projective normal surface with $K_S$ $\bQ$-Cartier, and $f: Z_{2,0} \rightarrow S$ is a birational morphism whose exceptional divisors have nonpositive discrepancies. 

By inequality \eqref{minmodelineq}, $(n-1)\pi|_{Z_2}^* (K_X) \geq g^* (K_{Z_{2,0}})$.  Taking the pushforward to $S$ gives $(n-1) \psi^* (K_{X_{\can}}) \geq K_S$ by the commutativity of the diagram and because $K_{Z_{2,0}}$ and $f^* (K_S)$ differ by exceptional divisors contracted by $f$.  Both $\psi^*(K_{X_{\can}})$ and $K_S$ are nef divisors, the latter by the assumption on discrepancies. 
Therefore,

\begin{equation}
\label{mainineq}
((n-1)\psi^*(K_{X_{\can}}))^2 \geq (n-1)\psi^*(K_{X_{\can}})  \cdot K_S \geq K_S^2.
\end{equation}

Combining the inequality (\ref{mainineq}) with $K_X^n \geq (\pi_{\can}^*(K_{X_{\can}})|_{Z_2})^2$, we get

$$(n-1)^2 K_X^n \geq ((n-1)\pi_{\can}^*(K_{X_{\can}})|_{Z_2})^2 = ((n-1)\psi^*(K_{X_{\can}}))^2 \geq K_S^2.$$

Moreover, the intersection number $K_S^2$ is determined by the volume of the minimal surface $Z_{2,0}$ and the divisors are contracted by $f$.  We may obtain sharper bounds on $K_X^n$ via this method than from only considering smooth models of $Z_2$ because often, $K_S^2 > \Vol(Z_2)$.
Therefore, we seek to factor $\pi_{\can}|_{Z_2}: Z_2 \rightarrow X_{\can}$ through contractions to certain singular surfaces $S$.  

The minimal model $Z_{2,0}$ of $Z_2$ is a fibration of genus 2 curves over $\bP^1$. The general fiber is a genus 2 curve $Z_1$.
We know by the adjunction formula that $Z_2$ satisfies $K_{Z_2} \geq (n-1)Z_1$.  Supposing $n \geq 3$, $p_g(Z_2) \geq h^0(Z_2,(n-1)Z_1) \geq n \geq 3$.  Under these assumptions, we have the following properties (cf. \cite[Proposition 2.9]{CCJ}):

\begin{prop}
\label{extremesurf}
\begin{equation}
\label{minvolS0}
    \Vol(Z_2)=(K_{Z_{2,0}})^2 \geq 
\begin{cases}
\frac{8n-16}{3}, & n \equiv 2 \bmod 3 \\
\frac{8n-14}{3}, & n \equiv 1 \bmod 3 \\
\frac{8n-12}{3}, & n \equiv 0 \bmod 3 \text{ .}
\end{cases}
\end{equation}
The fibration $Z_2 \rightarrow \bP^1$ factors through $Z_{2,0}$ and $K_{Z_{2,0}} = (n-1)Z_1 + G$, where $G$ is an effective divisor.

If we have equality in \eqref{minvolS0}, then $G = V + 2E$, where $V$ is a vertical divisor, $E \cong \bP^1$ is a section of the fibration, and $E^2 = -\frac{1}{3}(n+1 + V \cdot E)$.  The divisor $V$ satisfies $K_{Z_{2,0}} \cdot V = 0$ so that all irreducible components of $V$ are $(-2)$-curves contained in fibers.  Finally, $V \cdot E = 0, 1, 2$, in the cases $n = 3k+2$, $n = 3k+1$, and $n = 3k$, respectively.
\end{prop}

One can show that if $K_X^n$ is small enough to compete with the examples in Section \ref{n-1_ex}, then $\Vol(Z_2)$ also must be equal to the minimum in \eqref{minvolS0}.  Further, $\pi_{\can}|_{Z_2}: Z_2 \rightarrow X_{\can}$ factors through the minimal model $g: Z_2 \rightarrow Z_{2,0}$.  Using Proposition \ref{extremesurf}, we may identify some additional curves on $Z_{2,0}$ which are contracted by $\pi_{\can}$.  In each case, we will contract a connected curve of arithmetic genus $0$ with negative-definite self-intersection matrix.  This results in a projective surface with rational singularities (see \cite[Theorem 2.3]{Artin} and \cite[Proposition 1]{Artin2}), which is therefore $\bQ$-Gorenstein.

\smallskip

{\bf Case \ref{GeoP}.1.} $n=3k+2$, $k\geq 1$.

This is the easiest case, because we can define the appropriate $f: Z_{2,0} \rightarrow S$ by only contracting the curve $E$ of Proposition \ref{extremesurf} on $Z_{2,0}$. This gives $K_{Z_{2,0}}^2 = f^*K_S + aE$ with $a = \frac{-k+1}{k+1} \leq 0$.  Note that the $n = 2$ case must be treated separately.  Assuming that $K_{Z_{2,0}}^2$ is its minimum of $\frac{8n-16}{3}$, we may compute  

$$(n-1)^2K_X^n \geq K_S^2 = \frac{8n-16}{3} - (aE)^2 = \frac{3(n-1)^2}{n+1}.$$

Therefore, $K_X^n \geq \frac{3}{n+1}$ and we recover the same volume bound.

\medskip

We can see this curve $E$ explicitly in Example \ref{43}: the general weighted projective hypersurface
$$X_{10(k+1)} \subset \bP^{n+1}(1^{(n)},2(k+1),5(k+1))$$
has canonical dimension $n-1$, with the canonical map a rational map $X \dashrightarrow \bP^{n-1}$ having generic fiber a curve of genus 2.

Now, consider the intersection of $n-2$ general sections of the reflexive sheaf $K_X$.  This is a general surface $S = S_{10(k+1)} \subset \bP^3(1,1,2(k+1),5(k+1))$ with one singular point $p$, which is a quotient singularity of type $\frac{1}{k+1}(1,1)$.  This is a cone over the rational normal curve of degree $k+1$, so we can explicitly resolve it with a single blowup at $p$.  Denote this blowup as $f: Z_{2,0} \rightarrow S$. This has the same numerical properties as in the general case above.  From the description as a weighted projective hypersurface, we know $K_S^2 = \frac{(3k+1)^2}{k+1}$.  Therefore,
$$K_{Z_{2,0}}^2 = \frac{8n-16}{3},$$
so that the smooth model has volume matching the minimum in Proposition \ref{extremesurf}.  As the notation suggests, this is the same $Z_{2,0}$ obtained by desingularizing $X$ with a blowup and applying the machinery of Section \ref{set}.
\smallskip

{\bf Case \ref{GeoP}.2.} $n=3k+1$, $k\geq 1$.
If $K_{Z_{2,0}}^2$ achieves the minimum of $\frac{8n-14}{3}$, then by Proposition \ref{extremesurf}, $V \cdot E = 1$.  In general, $V$ could have several $(-2)$-curves as irreducible components, but a unique one, say $V_1$, intersects $E$ in a single point.  There is a birational morphism $f:Z_{2,0} \rightarrow S$ contracting exactly the connected curve $V_1 \cup E$.

The morphism $\pi_{\can}:Z_2 \rightarrow X_{\can}$ contracts both these curves on the minimal model, so it factors through $f: Z_{2,0} \rightarrow S$ and $K_{Z_{2,0}} = f^*K_S + aE + bV_1$, where $a = \frac{1-k}{2k+1}$ and $b = \frac{2-2k}{2k+1}$ are both nonpositive.  One can calculate
$$(n-1)^2 K_X^n \geq K_S^2 = \frac{8n-14}{3} - (aE + bV_1)^2 = \frac{6(n-1)^2}{2n+1},$$

so that $K_X^n \geq \frac{6}{2n+1}$, in agreement with Corollary \ref{41}.

\medskip

Once again, these exceptional curves appear in Example \ref{44}.  The general weighted projective hypersurface
$$X_{10k+6} \subset \bP^{n+1}(1^{(n)},2k+1,5k+3)$$
has canonical dimension $n-1$ with the canonical map a rational map $X \dashrightarrow \bP^{n-1}$ having generic fiber a curve of genus 2.  As above, we consider the intersection of $n-2$ general sections of $K_X$ to obtain $S = S_{10k+6} \subset \bP^3(1,1,2k+1,5k+3)$.  Then $K_S^2 = \frac{6(n-1)^2}{2n+1}$ and $S$ has only one singular point, at $[0:1:0:0]$, which is a quotient singularity of type $\frac{1}{2k+1}(1,5k+3) = \frac{1}{2k+1}(1,k+1)$.  The resolution of any surface quotient singularity of the form $\frac{1}{P}(1,Q)$ is related to the {\it Hirzebruch-Jung continued fraction} of $\frac{P}{Q}$ (see, for example, \cite[Section 2]{HTU}).  In this case,
$$\frac{2k+1}{k+1} = 2 - \frac{1}{k+1},$$
so $\frac{P}{Q} = [2,k+1]$ is the continued fraction representation.  This means that the exceptional locus has two rational curves as irreducible components, with self-intersections $-2$ and $-(k+1) = \frac{1}{3}(n+2)$, corresponding to $V_1$ and $E$ above, respectively.  It's straightforward to verify (e.g. using \cite[Section 2]{HTU}) that the discrepancies also agree with the above, so that 
$$K_{Z_{2,0}}^2 = \frac{8n-14}{3},$$
as expected. For $n = 4$, the singularity $\frac{1}{3}(1,2)$ is canonical (in particular, an $A_2$ singularity) and the resolution is crepant.  For $n \geq 7$, $S$ has worse than canonical singularities so $K_{Z_{2,0}}^2 < K_S^2$.

\smallskip

{\bf Case \ref{GeoP}.3.} $n=3k$, $k\geq 1$.
If $K_{Z_{2,0}}^2$ achieves the minimum of $\frac{8n-12}{3}$, then by Proposition \ref{extremesurf}, $V \cdot E = 2$.  Unlike in the previous two cases, these intersection numbers do not pin down a unique component graph for the connected component of $V_{\mathrm{red}} \cup E$ containing $E$.  The simplest configuration is that there are two irreducible components $V_1$ and $V_2$ of $V$ intersecting $E$, each in a single point, and no other components of $V$ intersecting these.  Fortunately, one can verify that the other possible configurations yield the same volume bound; suppose for clarity that we have the simple configuration.  Once again, there is a birational morphism $f: Z_{2,0} \rightarrow S$ contracting exactly $E \cup V_1 \cup V_2$ and $\pi_{\can}$ factors through this morphism.  $K_{Z_{2,0}} = f^*K_X + aE + bV_1 + cV_2$ with $a = c = \frac{1-k}{2k}$ and $b = \frac{1-k}{k}$ (both nonpositive).  We have

$$(n-1)^2 K_X^n \geq K_S^2 = \frac{8n-12}{3} - (aE + bV_1 + cV_2)^2 = \frac{3(n-1)^2}{n}.$$

This gives $K_X^n \geq \frac{3}{n}$, as desired.

\medskip

Example \ref{45} again displays the same behavior.  For any positive integer $n = 3k$, $k \geq 1$, the general weighted projective hypersurface
$$X_{20k+2} \subset \bP^{n+1}(1,2^{(n-1)},4k,10k+1)$$
has canonical dimension $n-1$.  It has $K_X = \mathcal{O}_X(2)$ and volume $\frac{1}{k} = \frac{3}{n}$.  The rational map $|K_X| = |\mathcal{O}_X(2)|$ has image $\bP^{n-1}(1,2^{(n-1)}) \cong \bP^{n-1}$.  The intersection of $n-2$ general sections of $K_X$ can be identified with a general $S_{20k+2} \subset \bP^3(1,2,4k,10k+1)$.  To see this, let $x_1$ have degree $1$ and $x_2, \ldots, x_n$ have degree $2$, so that a section of $\mathcal{O}_X(2)$ is a polynomial of the form $a_1 x_1^2 + a_2 x_2 + \cdots + a_n x_n$.  After a suitable change of coordinates within $x_2, \ldots x_n$, we may assume that the common zero set of $n-2$ general sections is given by equations $x_1^2 = x_3, x_1^2 = x_4,\ldots,x_1^2 = x_n$, which yields the result.

This surface $S$ satisfies $K_S^2 = \frac{3(n-1)^2}{n}$ and is smooth away from the weighted $\bP^1$ with weights $2$ and $4k$.  The general $S$ only intersects this stratum at the coordinate point of $4k$.  At this base point, there is a singularity of type $\frac{1}{4k}(1,2k+1)$.  The corresponding Hirzebruch-Jung continued fraction is
$$\frac{4k}{2k+1} = 2 - \frac{1}{(k+1) - \frac{1}{2}},$$
also represented as $\frac{4k}{2k+1} = [2,k+1,2]$.  Therefore, the minimal resolution of this singularity is a chain of three exceptional rational curves with self-intersections $-2$, $-(k+1) = \frac{1}{3}(n+3)$, and $-2$.  These play the role of $V_1, E$, and $V_2$ above, respectively.    Let $f: Z_{2,0} \rightarrow S$ be the resolution of singularities. Then 
$$K_{Z_{2,0}}^2 = \frac{8n-12}{3}.$$

This $Z_{2,0}$ is a minimal surface and admits a genus 2 fibration to $\bP^1$ as above.  When $n = 3$, $S$ is canonical, with a singularity of type $A_3$.  For $n \geq 6$, $S$ has worse-than-canonical singularities and $K_S^2 > K_{Z_{2,0}}^2$.  Table 1 summarizes the exceptional curves in each of the three cases.
\end{proof}

\begin{table}
\label{table}
\renewcommand{\arraystretch}{1.5}
\centering
\begin{tabular}{ccc}
\begin{tikzpicture}
\draw (0,0) -- node[below] {$E$} (3,0);
\end{tikzpicture} &
\begin{tikzpicture}
\draw (0,0) -- node[below] {$E$} (3,0);
\draw (0.5,-0.5) -- node[left] {$V_1$} (0.5,2.5);
\end{tikzpicture} &
\begin{tikzpicture}
\draw (0,0) -- node[below] {$E$} (3,0);
\draw (0.5,-0.5) -- node[left] {$V_1$} (0.5,2.5);
\draw (2.5,-0.5) -- node[right] {$V_2$} (2.5,2.5);
\end{tikzpicture}
\\

$n = 3k+2$ & $n=3k+1$ & $n = 3k$ \\
\end{tabular}
\caption{Exceptional curves of $Z_{2,0} \rightarrow S$ for optimal examples in the three cases}
\label{summary}
\end{table}

\bigskip

The results of this section notwithstanding, not all minimal varieties satisfying \eqref{Volmin} are necessarily quasi-smooth weighted projective hypersurfaces.  This is illustrated by the following example.

\begin{exmp}\label{exnonhyp} 
Let $n=3k$ for $k\geq 1$.  The general hypersurface
$$X_{10k+1}\subset \bP(1^{(3k)}, 2k, 5k)$$
is quasi-smooth and has only one non-canonical singularity of type $\frac{1}{5k}(1^{(3k-1)}, 2k)$, which satisfies the ``nefness criterion'' of Chen-Jiang-Li \cite[Theorem 1.3]{CJL}. So, after performing a weighted blowup at this point, by \cite[Proposition 2.9]{CJL} one gets a minimal variety $Y$ of dimension $n$ with canonical volume $\frac{1}{k}=\frac{3}{n}=\frac{6}{2n+(n \bmod 3)}$.  This example does not have $\mathrm{Pic}(X) \cong \bZ$, so it is not a quasi-smooth weighted projective hypersurface (see \cite[Theorem 3.2.4]{Dolgachev}).

We study the canonical stability index of this example. Denote by $x_i$ ($i=1,\ldots, 3k$) the coordinates of weight $1$, by $w$ the coordinate with weight $2k$, and by $t$ the coordinate with weight $5k$ in $\bP(1^{(3k)}, 2k, 5k)$.  The defining equation can be written:
$$f_0(x_1,\ldots,x_{3k},w)+f_1(x_1,\ldots,x_{3k},w)t+(a_1x_1+\cdots+a_{3k}x_{3k})t^2=0.$$
Denote by $\theta:Y \rightarrow  X:=X_{10k+1}$ the weighted blowup with weights $(1^{(3k-1)}, 2k)$.  Then we have $5kK_Y=\theta^*(5kK_X)-E$ with $E$ the exceptional divisor and 
$\theta(E)=Q\not\in (t=0)$.  This means that there are the following injective maps:
$$H^0(Y, 5kK_Y)\hookrightarrow \Lambda \hookrightarrow H^0(X, 5kK_X)$$
where $\Lambda$ is the vector space spanned by products of $x_1,\ldots,x_{3k},w$ of weighted degree $5k$. As $\Lambda$ does not give a birational map of $X$ (since $t$ is missing), $\varphi_{5k,Y}$ is non-birational. Hence $r_s(Y)=5k+1=\frac{5n+3}{3}=r_{n,n-1}$. 
\end{exmp}

\begin{rem}
In the explicit examples in Section \ref{n-1_ex} and Example \ref{exnonhyp}, the smooth models $Z_{2,0}$ of surfaces that appear are {\it Horikawa surfaces}, that is, they are on the Noether line $K_{Z_{2,0}}^2 = 2p_g(Z_{2,0}) - 4$.  For such varieties, $|K_{Z_{2,0}}|$ gives a ramified double cover to a ruled surface and the types of ruled surfaces and branch loci that appear are completely classified (see \cite{Horikawa}).  This suggests the following question:
\end{rem}

\begin{question}
If $X$ is a minimal variety of general type of dimension $n$, canonical dimension $n-1$, and smallest possible volume, is the minimal model of the associated general surface $Z_2$ a Horikawa surface?
\end{question}

\section*{\bf Acknowledgments}

M. Chen is a member of the Key Laboratory of Mathematics for Nonlinear Sciences, Fudan University. We would like to thank Chen Jiang and Burt Totaro for useful discussions and suggestions.

\end{document}